\documentclass[12pt,a4paper]{article}

\usepackage{algorithm}
\usepackage{algpseudocode}
\usepackage[utf8]{inputenc}
\usepackage[english]{babel}
\usepackage{color, colortbl}
\usepackage{amsmath}
\usepackage{amsthm}
\usepackage{amsfonts}
\usepackage{amssymb}
\usepackage{graphicx}
\usepackage{mathtools}
\usepackage[left=2cm,right=2cm,top=2cm,bottom=2cm]{geometry}
\usepackage{hyperref}

\newcommand{\init}{\mathrm{Init}}
\newcommand{\unsafe}{\mathrm{Unsafe}}
\newcommand{\RR}{\mathbb{R}}

\newcommand{\tbts}[3]{
\begin{bmatrix}
{#1} & {#2} \\
{#2}^T & {#3}
\end{bmatrix}
}

\definecolor{fc}{RGB}{255,255,255}
\definecolor{sc}{RGB}{255,255,255}

\newtheorem{lemma}{Lemma}

\title{Factorization of Saddle-point Matrices in Dynamical Systems Optimization---Reusing Pivots}
\author{Jan Ku\v{r}\'{a}tko\thanks{ORCID: \href{http://orcid.org/0000-0003-0105-6527}{0000-0003-0105-6527}; Faculty of Mathematics and Physics, Charles University, Czech Republic; Institute of Computer Science, The Czech Academy of Sciences}\ \thanks{This work was supported by the Czech Science Foundation (GACR) grant number  GA15-14484S with institutional support RVO:67985807.} }

\begin{document}
%		Title
\maketitle
%		Abstract
\abstract{
In this paper we consider the application of direct methods for solving a sequence of saddle-point systems. Our goal is to design a method that reuses information from one factorization and applies it to the next one. In more detail, when we compute the pivoted $LDL^T$ factorization we speed up computation by reusing already computed pivots and permutations. 
%TODO - work
We develop our method in the frame of dynamical systems optimization. Experiments show that the method improves efficiency over Bunch-Parlett while delivering the same results. \\[10mm]
\textbf{Keywords:} saddle-point matrix; symmetric indefinite factorization; dynamical systems; sequential quadratic programming 
%		OLD version
%In this paper we consider the application of direct methods for solving saddle-point systems that arise in the problem of optimization of dynamical systems. We describe the structure of nonzero elements in the factor $L$ in the $LDL^T$ factorization of the saddle-point matrix arising in such problems. Moreover, we develop a hybrid method that copes with the ill-conditioning of the upper left block of the saddle-point matrix. In the end we describe a monitoring strategy for the pivots selection in the Bunch-Parlett factorization. This will allow us to use already computed permutations in order to reduce the work spent on searching for pivots.
}
\section{Introduction}
\label{sec:Intro}
Consider a sequence of saddle-point systems arising, for example, in Sequential Quadratic Programming (SQP)~\cite{LuksanVlcek:2001, Nocedal:2006}, that is
%OLD - ver 2
%State of the art methods for solving the saddle-point systems are described in~\cite{Benzi:2005,Nocedal:2006}. One has a choice of using direct methods~\cite{Bunch:1971,Kaufman:1977}, iterative methods and also their combination. Therefore, one can choose an appropriate method with respect to properties of a saddle-point matrix $K$ and then solves the system $Kx = y$. However, our interest lies in solving a sequence of saddle-point systems arising, for example, in Sequential Quadratic Programming (SQP)~\cite{LuksanVlcek:2001, Nocedal:2006}, that is
\[
K_ix_i = y_i,\quad i = 1,2, \ldots
\]
where $K_i\in \RR^{n \times n}$, $x_i\in \RR^n$ and $y_i \in \RR^n$. Here, we assume matrices $K_i$ to have the same structure of nonzero entries. Based on this we study how to solve the system $K_ix_i = y_i$, benefiting from information from previous iterations.
%	OLD - ver 2
%If we were to compute the symmetric indefinite factorization $P_iK_iP_i^T = L_iD_iL_i^T$ in every iteration of SQP, then we would spend significant amount of time on searching matrices $K_i$ for pivots~\cite{Bunch:1971,Kaufman:1977}. 
Our goal is to develop a strategy that reduces the amount of work spent on searching the matrix for pivots in direct methods. 
%In other words, when finding a permutation~$P$ such that $PK_iP^T = L_iD_iL_i^T$ we reuse . In this paper we present a method that tries to find such~$P$. 
In more detail we describe our strategy in Section~\ref{sec:UpdatesBP} and compare its stability with Bunch-Parlett~\cite{Bunch:1971}.

% OLD - ver 2
%It may happen that reusing the same $P$ leads to a breakdown. We apply a pivot monitoring strategy and modify $P$ accordingly to avoid a breakdown. In more detail we describe our strategy in Section~\ref{sec:UpdatesBP} and compare its stability with Bunch-Parlett as well. 

%OLD
We develop this strategy in the frame of direct methods for solving saddle-point systems that arise in a certain class of optimization problems arising in verification of dynamical systems~\cite{Kuratko:2016}. Here, one seeks a solution of a dynamical system that originates in a given set of initial states and reaches another set of states  that are to be avoided, unsafe states. Saddle-point matrices $K_i$ that arise have specific structure which we try to exploit.

Such optimization problem occur for example in control and verification of hybrid systems~\cite{Branicky:2006,Chutinan:2003,Mohlman:2013,Zutshi:2013} and in motion planning~\cite{Lamiraux:2004,Plaku:2009}.  In addition, the techniques described in this paper apply to general underdetermined boundary value problems for ordinary differential equations~\cite{Ascher:1981}. Moreover, similar saddle-point matrices to ours arise, for example, in the mixed and hybrid finite element discretizations~\cite{Tuma:2002}, in a class of interior point methods~\cite{Vavasis:1994,Vavasis:1996,Wright:1992}, and time-harmonic Maxwell equations~\cite{Greif:2007,Hiptmair:2002}.

The saddle-point matrices that arise in such optimization problems and applications are sparse, whether one uses the SQP method~\cite{LuksanVlcek:2001} or the interior-point method~\cite{Luksan:2004}. Hence, direct methods for solving the saddle-point system look promising. However, the naive application of straightforward $LDL^T$ factorization often results in a failure due to ill-conditioning~\cite{Kuratko:2016,Vavasis:1994,Vavasis:1996,Wright:1992} and the singularity of the $(1,1)$ block~\cite{Greif:2007,Hiptmair:2002}.  
%In this paper we develop a direct method that successfully solves the saddle-point system. 

Denoting the saddle-point matrix and its factorization by
\begin{align*}
PKP^T & = P\tbts{H}{B}{-C}P^T = LDL^T,
\end{align*}
in more detail, the main contributions of this paper are: description of a strategy for selecting and reusing pivots in the symmetric and indefinite $PKP^T = LDL^T$ factorization; analysis of the growth factor in the reduced matrices; numerical comparison with Bunch-Parlett and Bunch-Kaufman on a series of benchmarks from dynamical systems optimization; description and exploitation of a specific structure of the saddle-point matrix $K$ and the factor $L$ in dynamical systems optimization problem~\cite{KuratkoRatschan:2016}; Alg.~\ref{Alg:S=Hybrid} that switches from unpivoted to pivoted factorization of the matrix $K$, balancing the speed with the stability of computation.
%	OLD
%the description of the banded structure of the factor $L$;  the description of the banded structure of the Schur complement $S = K/H$;  an algorithm for the computation of the $LDL^T$ factorization of the Schur complement $S$ without storing the whole matrix in memory. In addition, we develop a method that decides when to switch from $LDL^T$ with no pivoting strategies to the Bunch-Parlett method~\cite{Bunch:1971}. It is shown the saddle-point matrix has the upper left block $H$ ill-conditioned when we are close to a minimum~\cite{Kuratko:2016}, and that causes numerical problems during the computation. Therefore,  we aim at balancing the speed with the stability of computation. Furthermore, we describe the pivot monitoring strategy that allows us to reuse and update the permutation matrices in the Bunch-Parlett method.

The outline of the paper is as follows.
% We start with the problem formulation in section~\ref{sec:FormForm}. 
 In Section~\ref{sec:Motivation} we briefly review the optimization problem we try to solve~\cite{Kuratko:2016}. In Section~\ref{sec:SPM} we describe the structure of the saddle-point matrices. In Section~\ref{sec:LDLt} we will compute the $LDL^T$ factorization and prove that the factor $L$ has a banded structure of nonzero elements. Then the discussion about the implementation of the $LDL^T$ factorization follows and a hybrid method for solving the saddle-point system is described in Section~\ref{sec:Tweaks}. Sections~\ref{sec:UpdatesBP} contains a detailed description of reusing pivots and its effect on the stability of our method. Furthermore we include numerical results in Section~\ref{sec:BenchRes}. The whole paper is concluded with a summary and a brief discussion of results in Section~\ref{sec:Conclusion}. 
\section{Motivation}
\label{sec:Motivation}
Our motivation originates from the field of computer aided verification \cite{Branicky:2006,Lamiraux:2004,Zutshi:2013}. Consider a system of ordinary differential equations such that
\begin{equation}
	\label{eq:DiffEq}
	\dot{x}(t) = f(x(t)), \quad x(0) = x_0,
\end{equation}
where $x: \RR \to \RR^k$ is a function of variable $t \geq 0$, $x_0 \in \RR^k$ and $f: \RR^k\to\RR^k$ is  continuously differentiable. We denote the \emph{flow} of the vector field $f$ in~\eqref{eq:DiffEq} by $\Phi: \RR \times \RR^k \to \RR^k$ and for the fixed $x_0$ one has the solution $x(t)$ of~\eqref{eq:DiffEq}, where $x(t) = \Phi(t, x_0)$ for $t \geq 0$.

Denote the set of initial states by $\init$ and the set of states we try to avoid by $\unsafe$. Our goal is to find any solution $x(t)$ of~\eqref{eq:DiffEq} such that $x_0 \in \init$ and $\Phi(t_f, x_0) \in \unsafe$ for some $t_f > 0$, if it exists. 

In the previous work~\cite{Kuratko:2016} we solve this boundary value problem by the \emph{multiple-shooting} method~\cite{Ascher:1998}. That is, one computes a solution of~\eqref{eq:DiffEq} from shorter solution segments. Suppose we have $N$ solution segments of~\eqref{eq:DiffEq} such that their initial states are denoted by $x_0^i$ and their lengths by $t_i > 0$ for $1 \leq i \leq N$. Then the desired solution to our problem satisfies: $x_0^1 \in  \init$, $x_0^{i+1} = \Phi(t_i, x_0^i)$ for $1 \leq i \leq N-1$ (these are the matching conditions), and $\Phi(t_N, x_0^N) \in \unsafe$.

Boundary conditions $x_0^1 \in \init$ and $\Phi(t_N, x_0^N) \in \unsafe$ can be formulated either as equalities (points belong to the boundaries of the sets), or it can be given as inequalities (points are insides of sets). Either way there are infinitely many solutions~\cite{KuratkoRatschan:2016}, therefore, one needs to introduce a regularization. In the paper~\cite{KuratkoRatschan:2016} we formulate an objective function in the form $\sum t_i^2$, where $t_i$ is the length of the $i$-th solution segment, that drives the solution segments to have the same lengths.

In the end one solves a general nonlinear programming problem with $N(k+1)$ parameters, where those parameters are lengths of solution segments $t_i > 0$ and initial states $x_0^i \in \RR^k$, $1 \leq i  \leq N$. From now onwards we denote the number of parameters by $n = N(k+1)$ and the number of constraints by $m = (N-1)k + 2$.

%	This is a sidenote
%Note that one can work with a more general dynamical systems than in~\eqref{eq:DiffEq} such as \emph{switched} systems and hybrid dynamical systems~\cite{FehnkerIvancic:2004,Zutshi:2013}. To do that one needs to be able to compute or approximate~\cite{Betts:2010,Hiskens:2000} the sensitivity functions $\partial \Phi(t_i, x_0^i)/\partial x_0^i$, $1 \leq i \leq N$, in order to form the Jacobian of constraints $B$.
%
%
%
%
\section{Blocks of Saddle-point Matrix}
\label{sec:SPM}
The Line-search SQP method described in~\cite{Kuratko:2016} requires in each iteration the solution of the saddle-point system~\eqref{eq:KKTsystem}
\begin{equation}
\tbts{H}{B}{-C}
\begin{bmatrix}
x \\
y
\end{bmatrix}
=
\begin{bmatrix}
f \\
g
\end{bmatrix}
\quad \textrm{or} \quad
Ku = b\,,
\label{eq:KKTsystem}
\end{equation}
where $H \in \RR^{n \times n}$, $B \in \RR^{n \times m}$ with $n \geq m$ and $C \in \RR^{m \times m}$.

The elements of the matrix $B$ and $C$ are recomputed and the matrix $H$ is updated block by block by the \emph{BFGS} scheme. The structure of nonzero elements of the matrix $B$ in~\eqref{eq:HandB} remains the same throughout the iterations. 

%We consider the solution of a $2\times 2$ block linear system of the form
%\begin{equation}
%\tbts{H}{B}{-C}
%\begin{bmatrix}
%x \\
%y
%\end{bmatrix}
%=
%\begin{bmatrix}
%f \\
%g
%\end{bmatrix}
%\quad \textrm{or} \quad
%Ku = b\,,
%\label{eq:KKTsystem}
%\end{equation}
Since we are interested in the $LDL^T$ factorization of the saddle-point matrix $K$ that arises in the solution of reachability problems for dynamical systems~\cite{Kuratko:2016}, the blocks $H$ and $B$ have the form
\begin{equation}
\label{eq:HandB}
H
=
\begin{bmatrix}
H_1 & & \\
 & \ddots & \\
  & & H_N
\end{bmatrix}\,,
\quad
B
= 
\begin{bmatrix}
v & -M_1^T  & &   & & & \\
&-v_1^T &   &&  & & \\
&I &  - M_2^T & & & & \\
 & & -v_2^T & & & & \\
 & &  I &  & & & \\
& &   & \ddots&  & & \\
 & & & & & & \\
& &   &  &   & -M_{N-1}^T  & \\
& &  & &  &  -v_{N-1}^T & \\
& &  & &  & I &w\\ 
 & &  & &  & 0 & \beta
\end{bmatrix}\,.
\end{equation}
The matrix $H$ consists of blocks $H_i \in \RR^{(k+1) \times (k+1)}$, $1 \leq i \leq N$.  The matrix $C = diag(\gamma_1, 0, \ldots, 0, \gamma_2)$, where $\gamma_i \geq 0$ for $i = 1,2$. In the matrix $B$, there are blocks $M_i \in \RR^{k \times k}$ and vectors $v_i \in \RR^{k}$, $1 \leq i \leq N-1$. The matrix $I \in \RR^{k \times k}$ is the identity matrix of order $k$, the vectors $v$ and $w$ are nonzero and belong to $\RR^{k}$, and $\beta$ is a nonzero scalar. The first and the last columns of the matrix $B$ correspond to the boundary constraints $x_0^1 \in \init$ and $\Phi(t_N, x_0^N) \in \unsafe$.  

A similar banded structure of nonzero elements to the one in the matrix $B$ arises, for example, when one solves boundary value problems for ordinary differential equations~\cite[Sec.~5]{Ascher:1981}. However, we have additional entries $v_i$ in~\eqref{eq:HandB} because we consider the lengths of time intervals $t_i$ (the lengths of solution segments) to be parameters and not fixed values.

The saddle-point matrix $K$ satisfies the following conditions~\cite{Kuratko:2016}: the matrix $H$ is symmetric positive definite (\emph{BFGS} approximations of the Hessian), and the matrix $B$ has full column rank. Under these conditions the saddle point matrix $K$ is: nonsingular~\cite[Th.~3.1]{Benzi:2005}, indefinite~\cite[Th.~3.5]{Benzi:2005}, strongly factorizable~\cite[Th.~2.1]{Vanderbei:1995}.  

\section{$LDL^T$ Factorization}
\label{sec:LDLt}
In this section we give formulas for the $LDL^T$ factorization of the saddle-point matrix $K$ with blocks~\eqref{eq:HandB}. In addition, we describe the structure of nonzero elements of the unit lower triangular factor $L$ for which $K = LDL^T$. 

A standard approach~\cite[Ch.~4.1]{Golub:1996} to solving the linear system $Ku = b$, where $K$ is symmetric and nonsingular, is Alg.~\ref{Alg:$LDL^T$ Factorization}. 
\begin{algorithm}
\caption{Solve $Ku = b$ for $u$ by $LDL^T$ factorization}
\label{Alg:$LDL^T$ Factorization}
\begin{algorithmic}
\State Input: $K$ and the right-hand side vector $b$
\State Factorize the matrix $K = LDL^T$\Comment{Alg.~\ref{Alg:L}}
\State Solve $Lz = b$ for $z$ \Comment{Forward elimination}
\State Solve $Dw = z$ for $w$
\State Solve $L^Tu = w$ for $u$ \Comment{Back substitution}
\State Output: the solution $u$ to the system $Ku = b$
\end{algorithmic}
\end{algorithm}

Denote by $L_HD_HL_H^T$ the $LDL^T$ factorization of the matrix $H$ and by $L_SD_SL_S^T$ the $LDL^T$ factorization of the Schur complement $S = K/H$, where  $ S = -C - B^TH^{-1}B$. When forming the factor $L$ one factorizes $H = L_HD_HL_H^T$ and then $S = L_SD_SL_S^T$. Then
\begin{align}
\tbts{H}{B}{-C} & =
 \begin{bmatrix}
 L_H & 0 \\
 B^TL_H^{-T}D_H^{-1} & I
 \end{bmatrix}
 \begin{bmatrix}
 D_H & 0 \\
 0 & S
 \end{bmatrix}
 \begin{bmatrix}
 L_H^T & D_H^{-1}L_H^{-1}B \\
 0 & I
 \end{bmatrix}\, \notag \\
\label{eq:L} 
 & = 
  \begin{bmatrix}
 L_H & 0 \\
 B^TL_H^{-T}D_H^{-1} & L_S
 \end{bmatrix}
 \begin{bmatrix}
 D_H & 0 \\
 0 & D_S
 \end{bmatrix}
 \begin{bmatrix}
 L_H^T & D_H^{-1}L_H^{-1}B \\
 0 & L_S^T
 \end{bmatrix}\, \\
 & = LDL^T\,. \notag
\end{align}
Computation of the factor $L$ is summarized in Alg.~\ref{Alg:L} and follows the framework of the solution of \emph{equilibrium systems}~\cite[p.~170]{Golub:1996}.
\begin{algorithm}
\caption{Factorize the matrix $K = LDL^T$}
\label{Alg:L}
\begin{algorithmic}
\State Input: $K$
\State Factorize the matrix  $H = L_HD_HL_H^T$ \Comment{$LDL^T$ factorization of $H$ as in \eqref{eq:L_H}}
\State Solve $L_HD_HX = B$ for $X \in \RR^{n \times m}$
\State $S \leftarrow -C - X^TD_HX $ \Comment{Schur complement $S$}
\State Factorize the matrix $S = L_SD_SL_S^T$ \Comment{$LDL^T$ factorization of $S$ in Alg.~\ref{Alg:S=LDLt}}
\State $D \leftarrow diag(D_H, D_S)$
\State $L \leftarrow \begin{bmatrix}
L_H & 0 \\
X^T & L_S
\end{bmatrix}
$
\State Output: factors $L$ and $D$ such that $K =LDL^T$
\end{algorithmic}
\end{algorithm}

Matrix $H$ from~\eqref{eq:HandB} is block diagonal, then
\begin{equation}
\label{eq:L_H}
L_HD_HL_H^T
=
\begin{bmatrix}
L_{H, 1} & & \\
 & \ddots & \\
  & & L_{H,N}
\end{bmatrix}
\begin{bmatrix}
D_{H, 1} & & \\
 & \ddots & \\
  & & D_{H, N}
\end{bmatrix}
\begin{bmatrix}
L_{H, 1}^T & & \\
 & \ddots & \\
  & & L_{H,N}^T
\end{bmatrix},
\end{equation}
where $L_{H, i} \in \RR^{(k+1) \times (k+1)}$ is unit lower triangular and $D_{H, i} \in \RR^{(k+1) \times (k+1)}$ is diagonal, $1 \leq i \leq N$. 
\begin{lemma}
\label{lem:L_H}
Let $K = LDL^T$ be the $LDL^T$ factorization of the saddle-point matrix~\eqref{eq:KKTsystem} with blocks given in~\eqref{eq:HandB}. Then the $(1,1)$ block $L_H$ of the factor $L$ is block diagonal and its blocks $L_{H,i}\in \RR^{(k+1) \times (k+1)}$,$1 \leq i \leq N$, are unit lower triangular.
\end{lemma}
\begin{proof}
The matrix $H$~\eqref{eq:HandB} is symmetric positive definite, therefore, each $H_i$, $1 \leq i \leq N$, is symmetric positive definite. When one carries out the $LDL^T$ factorization of $H_i$, then one obtains factors $L_{H,i}$ and $D_{H,i}$ such that $H_i = L_{H,i}D_{H,i}L_{H,i}^T$ for $1 \leq i \leq N$. The $LDL^T$ factorization of $H$ can be written in the matrix form as in~\eqref{eq:L_H}, where $L_{H,i}$ is unit lower triangular and $D_{H,i}$ is diagonal with the $1 \times 1$ pivots on the diagonal, $1 \leq i \leq N$.
\end{proof}
We proceed with the computation of $B^TL_H^{-T}D_H^{-1}$ in the factor $L$ in~\eqref{eq:L}. 
%The result is formulated in Lemma~\ref{lem:BLD} and shows the banded structure of the block $BL_H^{-1}D_H^{-1}$.
\begin{lemma}
\label{lem:BLD}
Let $K = LDL^T$ be the $LDL^T$ factorization of the saddle-point matrix~\eqref{eq:KKTsystem} with blocks given in~\eqref{eq:HandB}. Then
\begin{equation}
\label{eq:BTL-TH-1}
B^TL_H^{-T}D_H^{-1} 
=
\begin{bmatrix}
s_1^T & & & & & \\
X_1 & Y_2  & & & & \\
 & X_2 & Y_3 & & & \\
  & & \ddots & &   & \\
 & & & X_{N-2}&  Y_{N-1} & \\
 & & & & X_{N-1} & Y_N  \\
 & & & & &  s_2^T
\end{bmatrix}\,,
\end{equation}
where 
\begin{align*}
s_1 & = D_{H,1}^{-1}L_{H,1}^{-1}\begin{bmatrix} v \\ 0 \end{bmatrix} \in \RR^{k+1}, \\ % [v^T\ 0]L_{H,1}^{-T}D_{H,1}^{-1} \in \RR^{1 \times (k+1)}\,,\\
X_i & = [M_i\ v_i]L_{H,i}^{-T}D_{H,i}^{-1} \in \RR^{k \times (k+1)}\,, \\
Y_i & = [I\ 0]L_{H,i}^{-T}D_{H,i}^{-1} \in \RR^{k \times (k+1)}\,,\\
s_2 & =  D_{H,N}^{-1}L_{H,N}^{-1}\begin{bmatrix} w \\ \beta \end{bmatrix} \in \RR^{k+1}, %  [w^T\ \beta]L_{H,N}^{-T}D_{H,N}^{-1} \in \RR^{1 \times (k+1)}\,,
\end{align*}
where $[v^T\ 0]^T \in \RR^{k+1}$, $[w^T\ \beta]^T \in \RR^{k+1}$, $[I\ 0] \in \RR^{k \times (k+1)}$ and $[M_i\ v_i] \in \RR^{k \times (k+1)}$ for $1 \leq i \leq N-1$.
\end{lemma}
\begin{proof}
The result follows from the direct computation of the matrix product $BL_H^{-T}D_H^{-1}$. The matrix $B$ is given in~\eqref{eq:HandB} and factors $L_H$ and $D_H$ in~\eqref{eq:L_H}. Since $L_H$ is block diagonal, then its inverse $L_{H}^{-1}$ is also block diagonal with blocks of the same size.
\end{proof}
To finish the description of the factor $L$ in~\eqref{eq:L} one needs to compute the Schur complement $S = K/H$ and factorize it. Lemma~\ref{lem:-S} shows the block $3$-diagonal structure of the Schur complement $S$.
\begin{lemma}
The Schur complement $S = -C - B^TH^{-1}B$ has the form
\label{lem:-S}
\begin{equation}
\label{eq:SchurCompl}
S
=
-
\begin{bmatrix}
\alpha_1 & w_1^T & & & & & \\
w_1 & V_1 & W_1^T & & & &\\
    &  W_1 & V_2 & & & \\
    &   &  & \ddots& & \\
    & & &  & V_{N-2} & W_{N-2}^T & \\
    & & & &  W_{N-2} & V_{N-1}& w_2 \\
    & & & &                  & w_2^T & \alpha_2
\end{bmatrix}\,,
\end{equation}
where
\begin{align*}
\alpha_1 & = 
[v^T\ 0]H_1^{-1}
\begin{bmatrix}
v \\
0
\end{bmatrix} + \gamma_1 \in \RR \,, \\
w_1 & = 
[M_1^T\ v_1]H_1^{-1}
\begin{bmatrix}
v \\
0
\end{bmatrix} \in \RR^{k}\,, \\
V_i & =
[M_i^T\ v_i]H_i^{-1}
\begin{bmatrix}
M_i \\
v_i^T
\end{bmatrix}
+
[I\ 0]H_{i+1}^{-1}
\begin{bmatrix}
I \\
0
\end{bmatrix} \in \RR^{k\times k}\,, \\
W_i & = 
[M_{i+1}^T\ v_{i+1}]H_{i+1}^{-1}
\begin{bmatrix}
I \\
0
\end{bmatrix} \in \RR^{k\times k}\,, \\
w_2 & = 
[I\ 0]H_N^{-1}
\begin{bmatrix}
w \\
\beta
\end{bmatrix} \in \RR^k \,, \\
\alpha_2 & = 
[w^T\ \beta]H_N^{-1}
\begin{bmatrix}
w \\
\beta
\end{bmatrix} + \gamma_2 \in \RR\,,
\end{align*}
with $[v^T\ 0]^T \in \RR^{k+1}$, $[w^T\ \beta]^T \in \RR^{k+1}$, $[I\ 0] \in \RR^{k \times (k+1)}$, $\gamma_1 \geq 0$ and $\gamma_2 \geq 0$ from the matrix $C$,  and $[M_i^T\ v_i] \in \RR^{k \times (k+1)}$ for $1 \leq i \leq N-1$.
\end{lemma}
\begin{proof}
The result follows from the matrix product $BH^{-1}B^T$, where $H^{-1}$ is a block diagonal matrix and $B$ is given in~\eqref{eq:HandB}. The matrix $C = diag(\gamma_1, 0, \ldots, 0, \gamma_2)$, therefore, it only affects the values $\alpha_1$ and $ \alpha_2$ in~\eqref{eq:SchurCompl}.
\end{proof}
The Schur complement $S$ is tri-block diagonal. We shall illustrate the process of $LDL^T$ factorization for $N = 3$. Then
\[
-S = 
\begin{bmatrix}
\alpha_1 & w_1^T & &  \\
w_1 & V_1 & W_1^T & \\
    &  W_1 & V_2 & w_2  \\
    & &w_2^T & \alpha_2  \\            
\end{bmatrix}
\longrightarrow 
\begin{bmatrix}
1 &  & &  \\
w_1/\alpha_1 & \hat{V}_1 & W_1^T & \\
    &  W_1 & V_2 & w_2  \\
    & &w_2^T & \alpha_2  \\            
\end{bmatrix}\,,
\]
where $\hat{V}_1 = V_1 - w_1w_1^T/\alpha_1$. Once we form the $LDL^T$ factorization of $\hat{V}_1 = \hat{L_1}\hat{D_1}\hat{L}_1^T$, then
\[
\begin{bmatrix}
1 &  & &  \\
w_1/\alpha_1 & \hat{V}_1 & W_1^T & \\
    &  W_1 & V_2 & w_2  \\
    & &w_2^T & \alpha_2  \\            
\end{bmatrix}
\longrightarrow
\begin{bmatrix}
1 &  & &  \\
w_1/\alpha_1 & \hat{L}_1 &  & \\
    &  W_1\hat{L}_1^{-T}\hat{D}_1^{-1} & \hat{V}_2 & w_2  \\
    & &w_2^T & \alpha_2  \\            
\end{bmatrix}\,,
\]
where $\hat{V}_2 = V_2 - W_1\hat{L}_1^{-T}\hat{D_1}^{-1}\hat{L}_1^{-1}W_1^T = V_2 - W_1\hat{V}_1^{-1}W_1^T$. Once more we form the $LDL^T$ factorization of $\hat{V}_2 = \hat{L}_2\hat{D}_2\hat{L}_2^T$, then
\[
\begin{bmatrix}
1 &  & &  \\
w_1/\alpha_1 & \hat{L}_1 &  & \\
    &  W_1\hat{L}_1^{-T}\hat{D}_1^{-1} & \hat{V}_2 & w_2  \\
    & &w_2^T & \alpha_2  \\            
\end{bmatrix}
\longrightarrow
\begin{bmatrix}
1 &  & &  \\
w_1/\alpha_1 & \hat{L}_1 &  & \\
    &  W_1\hat{L}_1^{-T}\hat{D}_1^{-1} & \hat{L}_2 &   \\
    & &w_2^T\hat{L}_2^{-T}\hat{D}_2^{-1} & \hat{\alpha}_2  \\            
\end{bmatrix}\,,
\]
where $\hat{\alpha}_2 = \alpha_2 - w_2^T\hat{L}_2^{-T}\hat{D}_2^{-1}\hat{L}_2^{-1}w_2 = \alpha_2 - w_2^T\hat{V}_2^{-1}w_2$. Finally, we put $\hat{\alpha}_2$ into the diagonal matrix $D_S$ and set the last elemental on the diagonal in $L_S$ to one. 
\begin{lemma}
\label{lem:L_s}
Let $S = L_SD_SL_S^{T}$ be the $LDL^T$ factorization of the Schur complement $S = -C - B^TH^{-1}B$, then
\begin{equation}
\label{eq:L_s}
L_S
=
\begin{bmatrix}
1 & & & & & \\
l_1& \hat{L}_1 & & & & \\
 & W_1\hat{L}_1^{-T}D_{S,1}^{-1} & \hat{L}_2 & & & \\
 & & \ddots & & & \\
 & & & W_{N-2}\hat{L}_{N-2}^{-T}D_{S,N-2}^{-1}  & \hat{L}_{N-1} & \\
 & & & & l_{2}^T & 1
\end{bmatrix}\,,
\end{equation}
where $\hat{L}_i \in \RR^{k\times k}$, $1 \leq i \leq N-1$, are unit lower-triangular, $l_1 = w_1/\alpha_1 \in \RR^k$, $l_2 = \hat{D}_{N-1}^{-T}\hat{L}_{N-1}^{-1}w_2 \in \RR^k$ and the diagonal matrix $D_{S}$ is such that
\begin{equation}
\label{eq:D_s}
-D_S
=
\begin{bmatrix}
d_1 & & & & \\
 & D_{S,1} & & & \\
 & & \ddots & & \\
 & & & D_{S,N-1} & \\
 & & & & d_N
\end{bmatrix}\,,
\end{equation}
where $d_1 \in \RR$, $d_2 \in \RR$ and $D_{S,i} \in \RR^{k \times k}$ for $1 \leq i \leq N-1$. Here $d_1 = \alpha_1$ from~\eqref{eq:SchurCompl} and $d_N = \alpha_2 - w_2^T\hat{L}_{N-1}^{-T}D_{S,N-1}^{-1}\hat{L}_{N-1}^{-1}w_2$. The scalar $\alpha_2$ and the vector $w_2$ come from~\eqref{eq:SchurCompl}.
\end{lemma}
Lemmas~\ref{lem:L_H},~\ref{lem:BLD} and~\ref{lem:L_s} describe the block structure of the factor $L$. Note that both matrices $B^TL_H^{-T}D_H^{-1}$ and $L_S$ are banded and their width is independent of $N$. In Fig.~\ref{fig:KKT_KKTldlt}, it is illustrated what the structures of nonzero entries of the saddle-point matrix $K$ and its factor $L$ are. 
%In this case constants were chosen such that $k = 10$ and $N = 40$~\cite[Sec.~6.3]{Kuratko:2016}.
\begin{figure}
\centering
\includegraphics[width=0.45\linewidth]{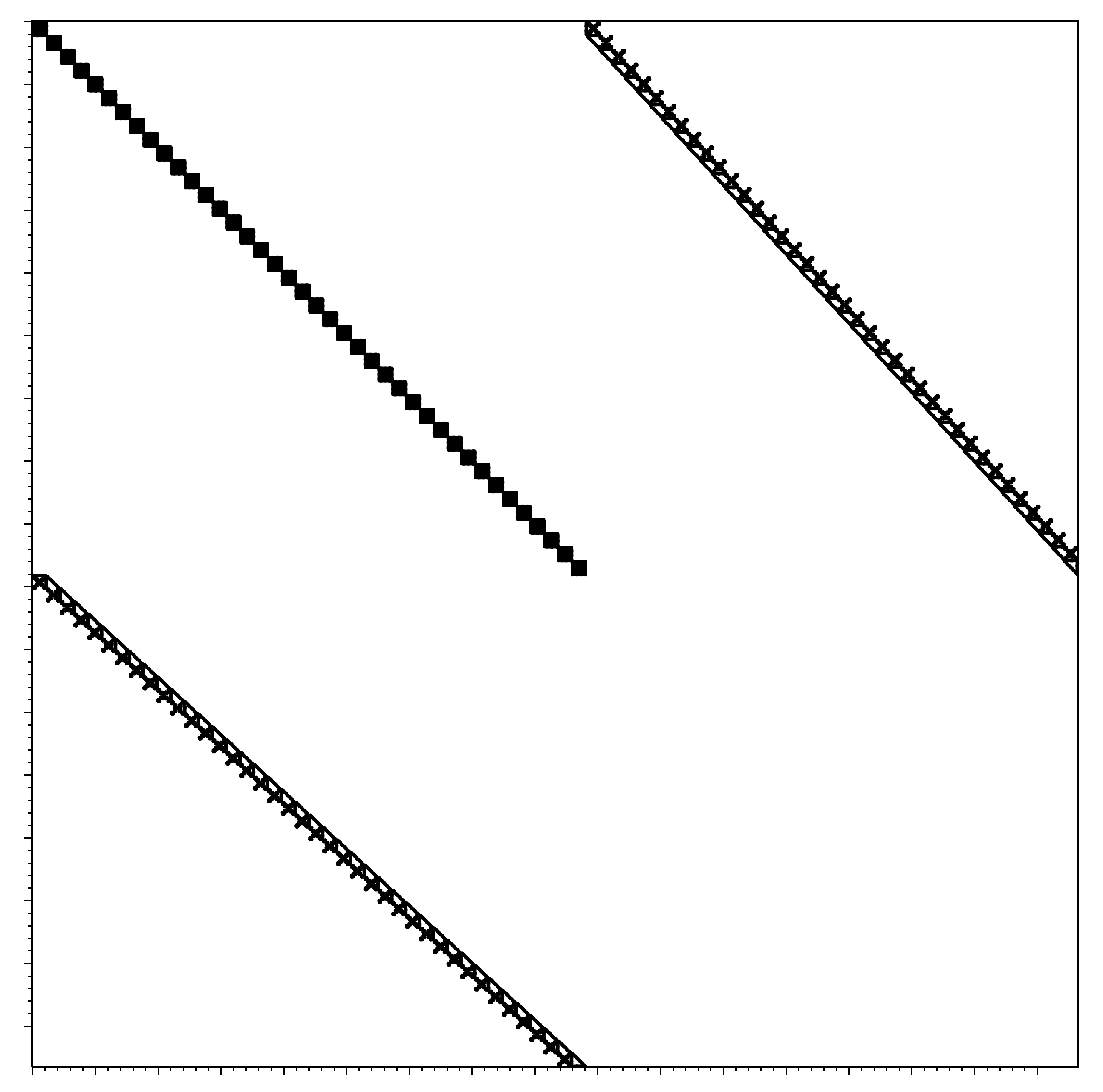}
\hfill
\includegraphics[width=0.45\linewidth]{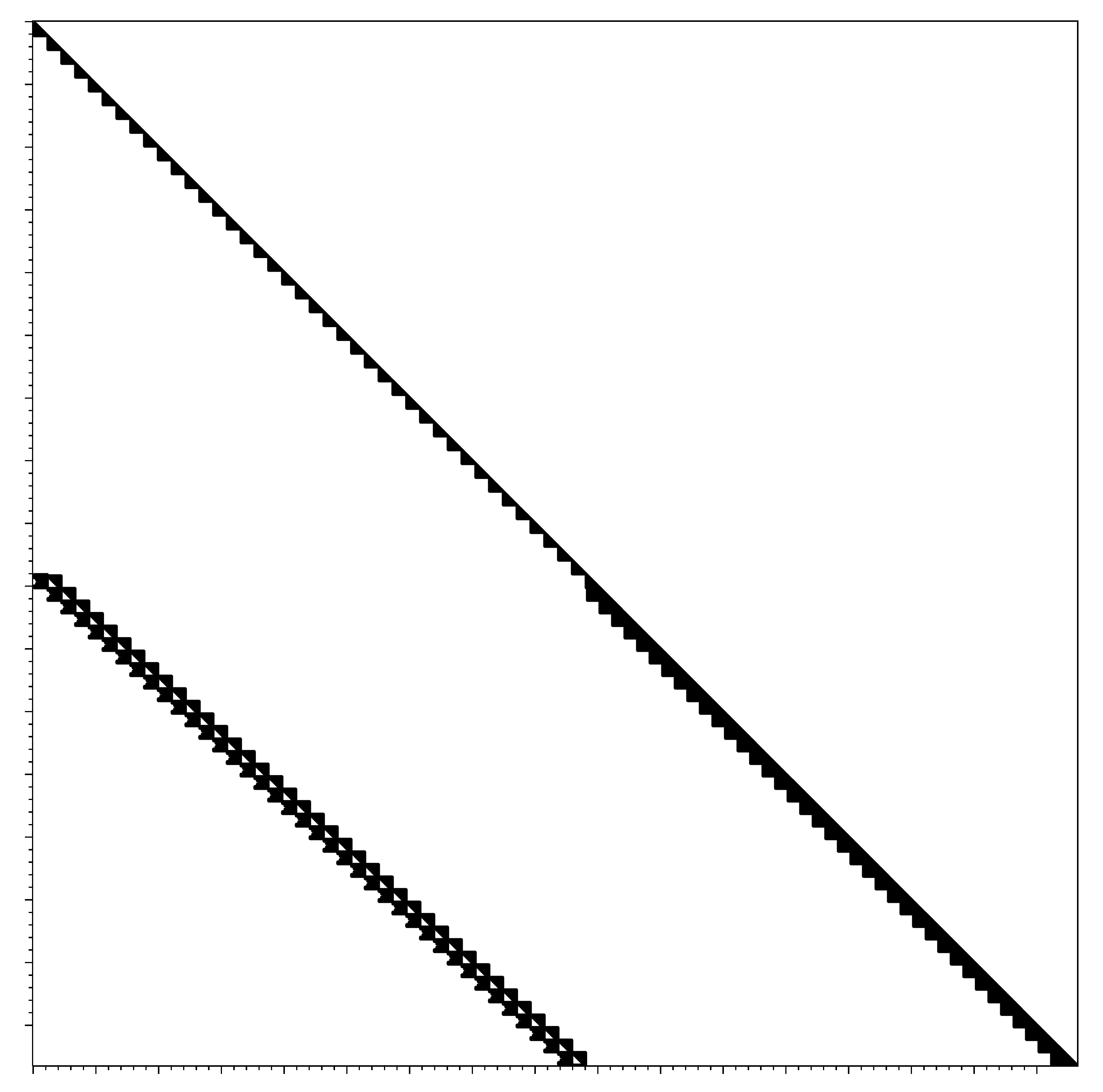}
\caption{The structure of nonzero entries: The saddle-point matrix $K$ on the left hand side and the unit lower triangular factor $L$ on the right hand side . For this instance, there are $9250$ nonzero elements in $K$ and $14 830$ nonzero elements in $L$. The dimension of matrices is $832 \times 832$.}
\label{fig:KKT_KKTldlt}
\end{figure}

We conclude this section by the observation that for the block $L_S$ the option for parallel computation of blocks $V_i$ and $W_i$ at the same time in~\eqref{eq:SchurCompl} is no longer available as in the case of the block $L_H$ in~\eqref{eq:L_H}. However, we do not need to keep the whole Schur complement $S$ in memory to get $L_S$ as it is shown in the discussion preceding Lemma~\ref{lem:L_s} and summarized in Alg.~\ref{Alg:S=LDLt}.
\begin{algorithm}
\caption{Factorize the matrix $S = L_SD_SL_S^T$}
\label{Alg:S=LDLt}
\begin{algorithmic}
\State Input: matrix $S$ of the form~\eqref{eq:SchurCompl}
\State $d_1 \leftarrow \alpha_1$ and $l_1 \leftarrow w_1/d_1$
\For{$i = 1$ to $N-1$} 
\If{$i=1$} 
\State Factorize the matrix $V_i - l_1d_1l_1^T = \hat{L}_1D_{S,1}\hat{L}_1^T$
\Else
\State Solve $\hat{L}_{i-1}D_{S, i-1}X = W_{i-1}^T$ for $X \in \RR^{k \times k}$ 
\State $\hat{L}_{i,i-1} \leftarrow X^T$ \Comment{Sub-diagonal block of $L_S$}
\State Factorize the matrix $V_i - \hat{L}_{i,i-1}D_{S,i-1}^{-1}\hat{L}_{i,i-1}^T = \hat{L}_iD_{S,i}\hat{L}_i^T$
\EndIf
\EndFor
\State Solve $\hat{L}_{N-1}D_{S, N-1}l_{2} = w_2$ for $l_{2} \in \RR^k$ 
\State $d_N \leftarrow \alpha_2 - l_{2}^TD_{S,N-1}^{-1}l_{2}$ 
\State $D_S \leftarrow diag(d_1,D_{S,1}, \ldots, D_{S,N}, d_N)$
\State Output: factors $L_S$ and $D_S$ such that $S = L_SD_SL_S^T$
\end{algorithmic}
\end{algorithm}

\section{Meeting the Ill-conditioned $H$}
\label{sec:Tweaks}
The accuracy of the computed solution $u$ of the saddle-point system~\eqref{eq:KKTsystem} with blocks from~\eqref{eq:HandB} by Alg.~\ref{Alg:$LDL^T$ Factorization} depends on the condition number of $H$~\cite[p.~171]{Golub:1996}. For the matrix $K$ is indefinite the condition number of $K$ may be much smaller than the condition number of $H$~\cite[p.~171]{Golub:1996}. Therefore, one may try to find a permutation matrix $P$ such that the factorization $PKP^T = LDL^T$ gives better numerical results. 

One can also find an ill-conditioned $(1,1)$ block in a class of interior point methods, electrical networks modelling and in the finite elements for a heat application~\cite{Vavasis:1994,Vavasis:1996,Wright:1992}. We also encountered the ill-conditioned $(1,1)$ block of $K$ in dynamical system optimization~\cite{Kuratko:2016}. Note that there are applications where the $(1,1)$ block is singular such as time-harmonic Maxwell equations~\cite{Greif:2007,Hiptmair:2002}, and linear dynamical systems in the paper~\cite{Kuratko:2016}.

One such approach to computation of $PKP^T = LDL^T$ is Bunch-Parlett~\cite{Bunch:1971}. However, this leads to a dense factor $L$ as illustrated in Fig.~\ref{fig:KKT_KKTbp}. In addition, finding elements for a pivoting strategy is very expensive and one needs to search a matrix for its maximal off-diagonal element. For a nonsingular symmetric matrix $A  \in \RR^{n \times n}$, the pivoting strategy requires between $n^3/12$ and $n^3/6$ comparisons~\cite[Sec.~6.3]{Bunch:1971}. On the other hand, the Bunch-Parlett factorization is nearly as stable as the \emph{LU} factorization with complete pivoting~\cite{Bunch:1971}. 

\begin{figure}
\centering
\includegraphics[width=0.45\linewidth]{kkt}
\hfill
\includegraphics[width=0.45\linewidth]{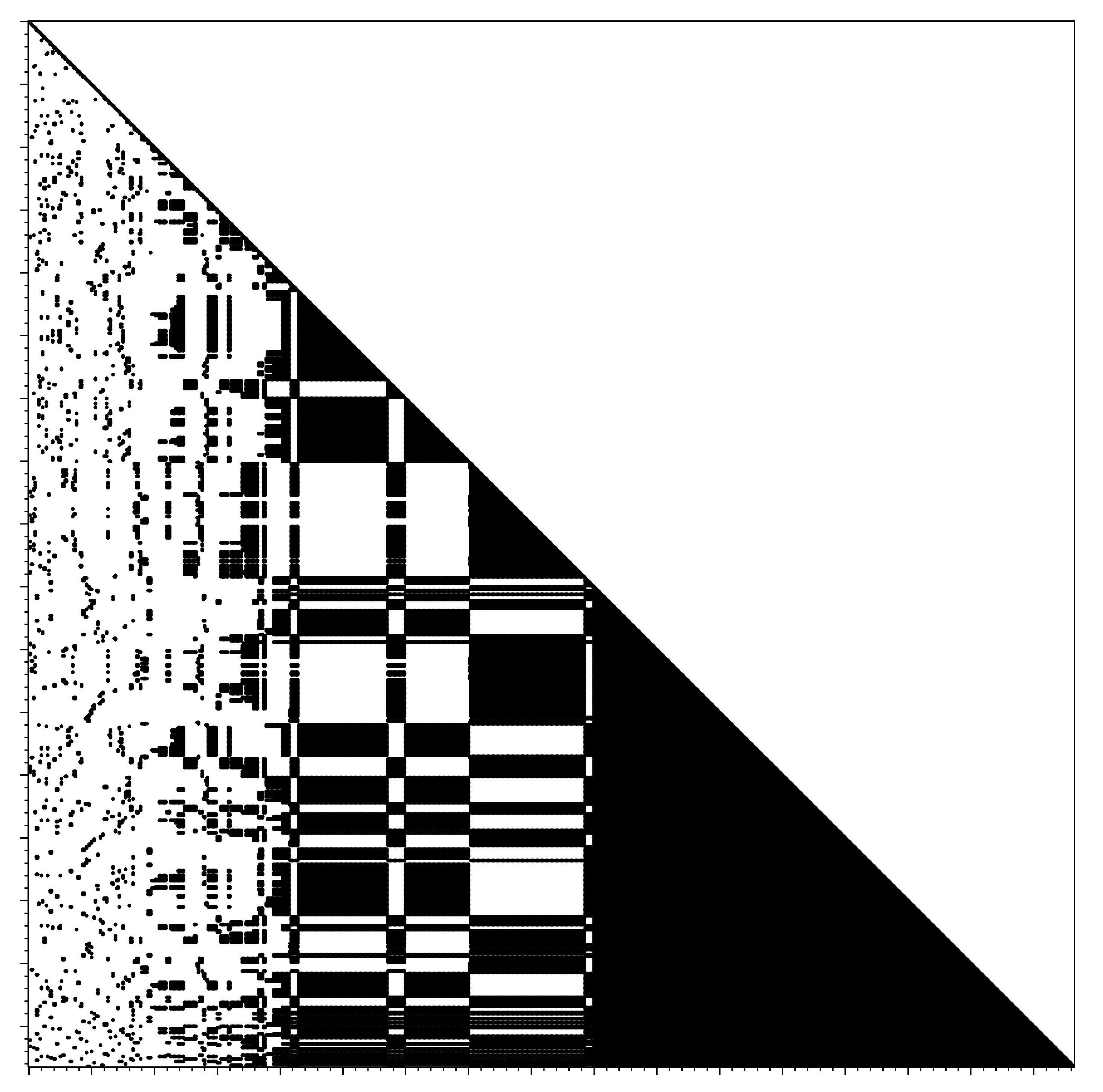}
\caption{The structure of nonzero entries: On the left hand side is the saddle-point matrix $K$ and on the right hand side is the unit lower triangular factor $L$, where $PKP^T = LDL^T$ by Bunch-Parlett. For this instance, there are $9250$ nonzero elements in $K$ and $146291$ nonzero elements in $L$. The dimension of matrices is $832 \times 832$.}
\label{fig:KKT_KKTbp}
\end{figure}

BFGS approximations of the matrix $H$~\eqref{eq:HandB} approach a singular matrix as the iteration process progresses~\cite{Kuratko:2016}. Therefore, the accuracy of solutions computed by Alg.~\ref{Alg:$LDL^T$ Factorization} deteriorates~\cite[p.~171]{Golub:1996}. 
%In order to improve the stability one can try to modify the matrix $D_H$ in the $LDL^T$ factorization of the matrix $H$. One approach is to change the diagonal matrix $D_H$ so that its condition number is less than $1/\varepsilon_H$, where $\varepsilon_H$ is a small constant. These considerations are summarized in Alg.~\ref{Alg:S=LDLt-stabilization}. However, we found it difficult to set the value for the parameter $\varepsilon_H$ and tried several values from the interval $[10^{-12}, 10^{-6}]$. The performance of Alg.~\ref{Alg:$LDL^T$ Factorization} is sensitive to the choice for $\varepsilon_H$.
%\begin{algorithm}
%\caption{Modification of the diagonal matrix $D_H$}
%\label{Alg:S=LDLt-stabilization}
%\begin{algorithmic}
%\State Input: matrix $D_H = diag(d_1, \ldots d_n)$, constant $\varepsilon_H$
%\For{$k = 1$ to $n$} 
%\If{$|d_k| < \varepsilon_H \| D_H \|$}
%\State $d_k \leftarrow sign(d_k) \varepsilon_H \| D_H \|$
%\Else
%\State leave $d_k$ unchanged
%\EndIf
%\EndFor
%\State Output: the diagonal matrix $D = diag(d_1, \ldots, d_n)$
%\end{algorithmic}
%\end{algorithm}
%
%Even with this modification the optimization method~\cite{Kuratko:2016}, using Line-search SQP, may still fail at computing the optimum. 
This happened to us in several cases when we used benchmark problems from~\cite{Kuratko:2016}. 
%Therefore, we do not use Alg.~\ref{Alg:S=LDLt-stabilization}.

However, we observed that the application of Bunch-Parlett instead of Alg.~\ref{Alg:$LDL^T$ Factorization} did not fail and delivered desired results. This leads us to the formulation of a hybrid method Alg.~\ref{Alg:S=Hybrid} that switches at some point from the straightforward $LDL^T$ factorization without pivoting to Bunch-Parlett. Our idea is to use Alg.~\ref{Alg:$LDL^T$ Factorization} as long as possible until the condition number of $D_H$ gets \emph{large}. When this behaviour is detected the method switches for Bunch-Parlett to finish. From numerical testing we found a suitable rule for switching to be that condition number of the diagonal matrix $D_H$ is greater than $1/\sqrt[3]{\varepsilon}$, where the machine precision $\varepsilon \approx 10^{-16}$.
\begin{algorithm}
\caption{Line-search SQP with the Pivoted and Unpivoted Factorization of $K$}
\label{Alg:S=Hybrid}
\begin{algorithmic}
\State Input: $K$ and the right hand side $b$, machine precision $\varepsilon$
\State difficult $\leftarrow False$, $H \leftarrow I$
\While{stopping criteria in \emph{Line-search SQP} are not met}
\If{difficult} \Comment{The matrix $H$ is ill-conditioned}
\State Factorize the matrix $PKP^T = LDL^T$ \Comment{Pivoted factorization}
\State Solve $Lw = Pb$ for $w$
\State Solve $Dz = w$ for $z$
\State Solve $L^Tu = z$ for $u$
\State $u \leftarrow  P^Tu$ \Comment{Permute elements in the solution $u$}
\State Update $H$, $B$ and $b$ from the solution $u$
\Else
\State Factorize $H = L_HD_HL_H^T$
\If{$\kappa(D_H) > 1/\sqrt[3]{\varepsilon}$} \Comment{Check the condition number of $D_H$}
\State difficult $\leftarrow True$
\Else
\State Solve $Ku = b$ for $u$ by the $LDL^T$ factorization \Comment{No pivoting}
\State Update $H$, $B$ and $b$ from the solution $u$
\EndIf
\EndIf
\EndWhile
\State Output: the solution $u$
\end{algorithmic}
\end{algorithm}

%One alternative to Bunch-Parlett and its selection of pivots in Alg.~\ref{Alg:S=Hybrid} is to use the Bunch-Kaufman method~\cite{Kaufman:1977}. There, one only needs $O(n^2)$ comparisons. This method is faster and less stable than Bunch-Parlett~\cite{Kaufman:1977}.
Our goal in the next section is to reduce the amount of work spent on searching for pivots in Bunch-Parlett. 
%For completeness, we compare both approaches in Section~\ref{sec:BenchRes}.

\section{Updating the Matrix $P$}
\label{sec:UpdatesBP}
% How about the permutation matrix P?
We do not need to compute a new permutation for the Bunch-Parlett method in every iteration and may try to use and update the one from the previous iteration  in Alg.~\ref{Alg:S=Hybrid}. Let $P_iK_iP_i^T = L_iD_iL_i^T$ be the Bunch-Parlett factorization of the saddle-point matrix $K_i$ from the $i$-th iteration in Alg.~\ref{Alg:S=Hybrid}. In the very next iteration we try to compute $P_iK_{i+1}P_i^T = L_{i+1}D_{i+1}L_{i+1}^T$, where $D_{i+1}$ has the same pattern of $1 \times 1$ and $2 \times 2$ pivots as the matrix $D_i$. Such a factorization of $K_{i+1}$ may not exist and we need to monitor the pivots and update the permutation matrix $P_i$ if necessary. 

In the paper~\cite{Sorensen:1977} there is an algorithm for updating a matrix factorization. However, the method works for matrices of the form $K_{i+1} = K_i + \sigma vv^T$, where $\sigma \in \RR$, $v \in \RR^{n+m}$ and $\sigma vv^T$ is a rank-one matrix. This is not the case in the problem we try to solve. 

We employ the following \emph{monitoring} strategy as we follow the pattern of pivots of $D_i$ and factorize the permuted saddle-point matrix $P_iK_{i+1}P_i^T$. If the $1 \times 1$ pivot $\beta$ satisfies $|\beta| > \varepsilon_1$, where $\varepsilon_1 > 0$, we use it and leave $P_i$ unchanged. If not, we apply the Bunch-Parlett method on the reduced matrix and update the permutation. In the case of the $2 \times 2$ pivot $\beta = \begin{bmatrix} a & b \\ b & c \end{bmatrix}$ we accept it if $| ac - b^2 | > \varepsilon_1$ and $\| \beta \| < \varepsilon_2$. If these conditions do not hold we apply the Bunch-Parlett method on the reduced matrix and update the permutation $P_i$.

The $1 \times 1$ pivot is useful if and only if $|\beta|$ is large relative to the largest off-diagonal element in absolute value~\cite[Sec.~4.2]{Bunch:1971}. Therefore, we only bound its modulus from below and the larger it is the better. However, it may happen that $|\beta|$ is small compared to off-diagonal elements and that causes the increase in the magnitudes of elements in the factor $L$ and the reduced matrix. That is the reason for introducing the second condition on the $2 \times 2$ pivots for which we require $\| \beta \| < \varepsilon_2$. Since we divide by the determinant of the  $2 \times 2$ pivot $\beta$ we need it to be bounded away from zero~\cite[Sec.~4.3]{Bunch:1971}. With this monitoring strategy one is not restricted to Bunch-Parlett and we also tried it for the Bunch-Kaufman method. 
%In the next section we show the results for two benchmark problems in Tab.~\ref{tab:ODE01HYBandMON},~\ref{tab:ODE07HYBandMON} and the amount of work reduction for constructing the permutation matrices $P$. However, the monitoring of pivots as presented is only a heuristic. 
%We do not have at hand the stability analysis of such modified Bunch-Parlett and Bunch-Kaufman methods.

In the rest of this section, we compare the growth factor of the elements in the reduced matrices for the pivot monitoring strategy with Bunch-Parlett~\cite{Bunch:1971}. We use the same notation as in that paper. 
%We assume that the matrix $A$ has columns and rows permuted and we compute the symmetric factorization. 
Let $A \in \RR^{n \times n}$ be symmetric and nonsingular such that
\[
A
=
\begin{bmatrix}
A_{1,1} & A_{1,2} \\
A_{1,2}^T & A_{2,2}
\end{bmatrix},
\]
where $A_{1,1} \in \RR^{j\times j}$, $c \in \RR^{j \times (n-j)}$ and $A_{2,2} \in \RR^{(n-j)\times (n-j)}$. If $A_{1,1}^{-1}$ exists, then
\[
A
=
\begin{bmatrix}
I_j & 0 \\
A_{1,2}^TA_{1,1}^{-1} & I_{n-j}
\end{bmatrix}
\begin{bmatrix}
A_{1,1} & 0 \\
0 & A_{2,2} - A_{1,2}^TA_{1,1}^{-1}A_{1,2}
\end{bmatrix}
\begin{bmatrix}
I_j & A_{1,1}^{-T}A_{1,2} \\
0 & I_{n-j}
\end{bmatrix},
\]
where $I_j$ is the identity matrix of order $j$, and $I_{n-j}$ or order $n-j$ respectively. The elements of the matrix $M = A_{1,2}^TA_{1,1}^{-1}$ are called multipliers and we consider $j = 1$ or $2$. We denote by $A^{(n)} = A$ and let $A^{(k)}$ be the reduced matrix of order $k$. In the end, let $\mu_0 = \max_{i,j}\{ |a_{i,j}|\, ; a_{i,j} \in A\}$ and $\mu_1 = \max_i\{ |a_{i,i}|\, ; a_{i,i} \in A \}$.

Suppose that the pivot $\beta = A_{1,1}$ is of order $1$, that is $j = 1$. Under our monitoring strategy we accept $\beta$ for the pivot if $| \beta | > \varepsilon_1$. Then the reduced matrix is $A^{(n-1)} = A_{2,2} - A_{1,2}^T\beta^{-1}A_{1,2}$.
\begin{lemma}
\label{lem:1b1pivotbound}
Let $\varepsilon_1 \in (0,1)$. If $|\beta| > \varepsilon_1$, then
\begin{align*}
m := \max_i\{ |m_i|\,; m_i \in M\} & \leq \frac{\mu_0}{\varepsilon_1}, \\
\mu_0^{(n-1)} := \max_{i,j}\{ |a_{i,j}|\ ; a_{i,j} \in A^{(n-1)}\} & \leq \left( 1 + \frac{\mu_0}{\varepsilon_1} \right)\mu_0.
\end{align*}
\begin{proof}
We follow~\cite[Lem.~1]{Bunch:1971} and replace $\mu_1$ with our lower bound $\varepsilon_1$ on the magnitude of the pivot $\beta$.
\end{proof}
\end{lemma}
We observe in Lemma~\ref{lem:1b1pivotbound} that the bound is more pessimistic than the bound in~\cite[Lem.~1]{Bunch:1971} for Bunch-Parlett. In more detail the bound $m < 1.562$, see~\cite[Lem.~5]{Bunch:1971}, on multipliers may not hold under our monitoring strategy. The reason behind this is that we take the pivot $\beta$ even if $|\beta| < \mu_0(1+\sqrt{17})/8$ as long as $|\beta| > \varepsilon_1$. Therefore, during the factorization the elements in the reduced matrix may grow in magnitude rapidly.

Suppose the pivot $\beta$ is of order $2$, that is $j = 2$. We accept $\beta$ if $|\det \beta| > \varepsilon_1$ and $\| \beta \| < \varepsilon_2$. Then the matrix $M \in \RR^{(n-2)\times 2}$.
\begin{lemma}
\label{lem:2b2pivotbound}
Let $\varepsilon_1 \in (0,1)$. If $|\det \beta| > \varepsilon_1$, then
\begin{align*}
m := \max_{i,j}\{ |m_{i,j}|\,; m_{i,j} \in M\} & \leq \frac{\mu_0(\mu_0 + \mu_1)}{\varepsilon_1}, \\
\mu_0^{(n-2)} := \max_{i,j}\{ |a_{i,j}|\ ; a_{i,j} \in A^{(n-2)}\} & \leq \left( 1 + \frac{2\mu_0(\mu_0 + \mu_1)}{\varepsilon_1} \right)\mu_0,\\
\varepsilon_1 < |\det \beta | & \leq \mu_0^2 + \mu_1^2.
\end{align*}
\begin{proof}
The first and the second inequality follows from~\cite[Lem.~2]{Bunch:1971}, where we use the lower bound $\varepsilon_1$ instead of $\det \beta$ in the denominator. The last chain of inequalities follows partially from our assumption that  $|\det \beta| > \varepsilon_1$ and~\cite[Lem.~3]{Bunch:1971}.
\end{proof}
\end{lemma}
Note that we accept the pivot $\beta$ even if $\mu_1 > \mu_0$ as long as  $|\det \beta| > \varepsilon_1$ and $\| \beta \| < \varepsilon_2$ hold.  Due to this fact we cannot derive the bound $\mu_0^{(k)} < (2.57)^{n-k}\mu_0$ as in the paper~\cite{Bunch:1971}. In our case the bound $m$ on multipliers depends on $1/\varepsilon_1$, therefore, the overall bound on the growth of elements in the reduced matrix $A^{(k)}$ will contain powers of $1/\varepsilon_1$. Because of this very reason we place the condition $\| \beta \| < \varepsilon_2$ on the $2\times2$ pivot trying to meet the undesired growth of elements in the reduced matrices.

A less expensive method than Bunch-Parlett is Bunch-Kaufman~\cite{Kaufman:1977}, since it requires only $O(n^2)$ comparisons when searching the matrix for pivots. It is accepted as the algorithm of choice when solving symmetric indefinite linear systems. Similarly to Bunch-Parlett one can show that the growth of elements in the reduced matrices is bounded~\cite{Kaufman:1977}, however, this time there is no bound on the entries of the factor $L$~\cite[Sec.~2.2]{Ashcraft:1998}. Therefore, it gives lower accuracy and can even be unstable~\cite{Ashcraft:1998}. There are ways around this as described in~\cite{Ashcraft:1998}, however, then the modified Bunch-Kaufman requires a higher number of comparisons lying somewhere between the number of comparisons of  Bunch-Kaufman and Bunch-Parlett.

Our proposed heuristic also does not bound entries in the factor $L$, however, it tries to skip the search for pivots. In Section~\ref{sec:BenchRes} we compare the pivot monitoring strategy with Bunch-Parlett and Bunch-Kaufman on a series of benchmarks.

\section{Computational Experiments}
\label{sec:BenchRes}
In this section we apply our method~\ref{Alg:S=Hybrid} to two benchmark problems from the paper~\cite{Kuratko:2016}. We test Alg.~\ref{Alg:S=Hybrid} with and without the monitoring strategy for updating the permutation matrix and compare the results. Both following benchmark problems are described in detail in the paper~\cite{Kuratko:2016}. In this paper, we consider only equality constraints, hence, the block $C$ in~\eqref{eq:KKTsystem} is the zero matrix. For the reader's convenience we describe the governing differential equations of those dynamical systems.

The first benchmark problem~\cite[Sec.~6.2]{Kuratko:2016} is a linear dynamical system given by
\[
\dot{x} = 
\begin{bmatrix}
\begin{bmatrix}
0 & 1 \\
-1 & 0
\end{bmatrix}
 & & \\
 & \ddots & \\
 & &
 \begin{bmatrix}
0 & 1 \\
-1 & 0
\end{bmatrix}
\end{bmatrix}x,
\]
where the statespace dimension is $k \in \mathbb{N}$. It was shown in~\cite{Kuratko:2016} that the Hessian matrix is singular and observed that the BFGS approximation of the Hessian approaches a singular matrix. This leads to a saddle-point matrix $K$ that has an ill-conditioned $(1,1)$ block.

The second benchmark problem~\cite[Sec.~6.3]{Kuratko:2016} is a nonlinear dynamical system such that
\[
\dot{x} = 
\begin{bmatrix}
\begin{bmatrix}
0 & 1 \\
-1 & 0
\end{bmatrix}
 & & \\
 & \ddots & \\
 & &
 \begin{bmatrix}
0 & 1 \\
-1 & 0
\end{bmatrix}
\end{bmatrix}x
+
\begin{bmatrix}
\sin(x_k) \\
\vdots \\
\sin(x_1)
\end{bmatrix},
\]
where the statespace dimension is $k \in \mathbb{N}$. Similar to the benchmark above the BFGS approximations of the Hessian approach singular matrix~\cite{Kuratko:2016}. 

In both benchmark problems the sets $\init$ and $\unsafe$ are balls of radius $1/4$. The stopping criteria on the norm of the gradient, the norm of the vector of constraints, the maximum number of iterations and the minimal step-size are the same as in the paper~\cite{Kuratko:2016}. For our monitoring strategy in Alg.~\ref{Alg:S=Hybrid} we set $\varepsilon_1 = 10^{-3}$ and $\varepsilon_2 = 10^6$. The results for the first benchmark problem are shown in the Tab.~\ref{tab:ODE01HYBandMON} and for the second benchmark problem in the Tab.~\ref{tab:ODE07HYBandMON} respectively. In all instances we were able to find a desired solution from $N$ solution segments for which $x_0^1 \in \init$ and $\Phi(\sum t_i, x_0^1) \in \unsafe$.

Both tables~\ref{tab:ODE01HYBandMON} and~\ref{tab:ODE07HYBandMON} have four parts. The first part consists of two columns denoted by $k$ -- the statespace dimension of the dynamical system and by $N$ -- the number of solution segments. The second part corresponds to Alg.~\ref{Alg:S=Hybrid} with no monitoring of pivots and there are three columns: $\#IT$ -- the number of iterations in Alg.~\ref{Alg:S=Hybrid}, $\#LDL^T$ -- the number of straightforward $LDL^T$ factorizations, and $\#B$-$P$/$B$-$K$ -- the number of Bunch-Parlett/Bunch-Kaufman factorizations. The third part shows the results of Alg.~\ref{Alg:S=Hybrid} with the monitoring of pivots. The meaning of columns $\#IT$ and $\#LDL^T$ remains the same. However, the column denoted by ``$\#$upd of $P$'' shows how many times the matrix $P$ was computed and updated. 

In the end, the last column denoted by $R$ gives the ratio of the number of Bunch-Parlett/Bunch-Kaufman factorizations to the number of updates of $P$ using the pivot monitoring, that is
\[
R :=\left\lfloor
\frac{\text{$\#B$-$P$/$B$-$K$}}{\text{$\#$upd of $P$}}
\right\rfloor .
\]
One can interpret $R$ in the following way. As $R$ approaches $\#B$-$P$/$B$-$K$, then Alg.~\ref{Alg:S=Hybrid} reuses pivots almost all the time. Especially, when $R = \#B$-$P$/$B$-$K$, then we carry the search for pivots only once. When $R$ approaches $1$, then Alg~\ref{Alg:S=Hybrid} searches $K$ for pivots more frequently, and ultimately when $R = 1$ then it uses standard Bunch-Parlett/Bunch-Kaufman throughout. 

It may happen that the sum of numbers from $\#LDL^T$ and $\#B$-$P$/$B$-$K$ columns is greater than the number in the column $\#IT$ because of the restarts in the LS-SQP method~\cite{Kuratko:2016}. Whenever there is a single value in a column, then both Bunch-Parlett and Bunch-Kaufman yield the same results. If the values in columns $\#IT$ in one row differ we do not compute $R$. The same applies when only the $LDL^T$ factorization with no pivoting was used.

We read the results in Tab.~\ref{tab:ODE01HYBandMON} and~\ref{tab:ODE07HYBandMON} in the following way. For example the last row of Tab.~\ref{tab:ODE01HYBandMON} is: the statespace dimension $k = 40$ and the number of solution segments $N = 30$ result in the optimization problem with $30(40+1) = 1230$ parameters and $(30-1)40+2 = 1162$ equality constraints. Then the saddle-point matrix $K$ is of order $2392$; Alg.~\ref{Alg:S=Hybrid} with no monitoring of pivots took $59$ iterations from which the matrix $K$ was factorized $9$ times by $LDL^T$ with no pivoting and $50$ times by Bunch-Parlett/Bunch-Kaufman. When we used Alg.~\ref{Alg:S=Hybrid} with the monitoring of pivots the matrix $P$ was computed once in the $10$th iteration. From this point onwards it was updated twice, therefore, the matrix $P$ was reused $47$ times. The ratio $R$ is then $\lfloor50:3\rfloor = 16$.

\begin{table}[h!]
\centering
\newcolumntype{g}{>{\columncolor{fc}}c}
\newcolumntype{j}{>{\columncolor{sc}}c}
\begin{tabular}{*{2}{g}|*{3}{g}||*{3}{j}|c}
$k$ & $N$ & $\#IT$ & $\#LDL^T$ & $\#B$-$P$/$B$-$K$ & $\#IT$ & $\#LDL^T$ & $\#$upd of $P$ & $R$ \\
\hline
 $10$  &  $5$  &  $47$  &  $5$  &  $42$  &  $47$  &  $5$  &  $1$ & $42$\\
  &  $10$  &  $26$  &  $19$  &  $7$ &  $26$  &  $19$  &  $1$ & $7$\\
  &  $15$  &  $31$  &  $17$  &  $14$  &  $31$  &  $17$  &  $1$ & $14$\\
  &  $20$  &  $400$  &  $3$  &  $418/408$  &  $111/258$  &  $3$  &  $2/5$ & -\\
  &  $25$  &  $45$  &  $17$  &  $28$  &  $45$  &  $17$  &  $1$ & $28$ \\
  &  $30$  &  $131/125$  &  $11$  &  $120/116$  &  $89/78$  &  $11$  &  $8/4$ & - \\ 
 $20$  &  $5$  &  $25$  &  $27$  &  $0$  &  $25$  &  $27$  &  $0$ & -\\
  &  $10$  &  $27$  &  $27$  & $0$  &  $27$  &  $27$  &  $0$ & - \\
  &  $15$  &  $24$  &  $9$  &  $15$  &  $24$  &  $9$  &  $1$ &  $15$\\
  &  $20$  &  $23$  &  $3$  &  $20$  &  $23$  &  $3$  &  $1/3$ & $20/6$\\
  &  $25$  &  $27$  &  $5$  &  $22$  &  $27$  &  $5$  &  $1/2$ & $22/11$\\
  &  $30$  &  $34$  &  $20$  &  $14$  &  $34$  &  $20$  &  $1$ & $14$\\
 $30$  &  $5$  &  $36$  &  $36$  &  $0$  &  $36$  &  $36$  &  $0$ & - \\
 &  $10$  &  $24$  &  $24$  &  $0$  &  $24$  &  $24$  &  $0$ & - \\
 &  $15$  &  $26$  &  $10$  &  $16$  &  $26$  &  $10$  &  $1$ & $16$\\
 &  $20$  &  $28$  &  $23$  &  $5$  &  $28$  &  $23$  &  $1$ & $5$\\
 &  $25$  &  $24$  &  $10$  &  $14$  &  $24$  &  $10$  &  $1$  & $14$\\
 &  $30$  &  $24$  &  $6$  &  $18$  &  $24$  &  $6$  &  $1$ & $18$\\
 $40$  &  $5$  &  $41$  &  $41$  &  $0$  &  $41$  &  $41$  &  $0$ & -\\
  &  $10$  &  $24$  &  $24$  &  $0$  &  $24$  &  $24$  &  $0$ & - \\
  &  $15$  &  $24$  &  $3$  &  $21$  &  $24$  &  $3$  &  $1/2$ & $21/10$\\
  &  $20$  &  $20$  &  $5$  &  $15$  &  $20$  &  $5$  &  $1$ & $15$\\
  &  $25$  &  $37$  &  $11$  &  $26$ &  $37$  &  $11$  &  $1$ & $26$\\
  &  $30$  &  $59$  &  $9$  &  $50$  &  $59$  &  $9$  &  $3$ & $16$\\  
\end{tabular}
\caption{The results for the linear benchmark problem~\cite[Sec.~6.2]{Kuratko:2016}. On the left hand side there are the parameters, in the middle the results of Alg.~\ref{Alg:S=Hybrid} with no monitoring of pivots and on the right hand side the results of Alg.~\ref{Alg:S=Hybrid} with the monitoring.}
\label{tab:ODE01HYBandMON}
\end{table}
\begin{table}[h!]
\centering
\newcolumntype{g}{>{\columncolor{fc}}c}
\newcolumntype{j}{>{\columncolor{sc}}c}
\begin{tabular}{*{2}{g}|*{3}{g}||*{3}{j}|c}
$k$ & $N$ & $\#IT$ & $\#LDL^T$ & $\#B$-$P$/$B$-$K$ & $\#IT$ & $\#LDL^T$ & $\#$upd of $P$ & $R$ \\
\hline
 $10$  &  $5$  &  $40$  &  $40$  &  $0$  &  $40$  &  $40$  &  $0$ 		& -	\\
   &  $10$  &  $246$  &  $175$  &  $71$  &  $246$  &  $175$  &  $1$ & $71$	\\
   &  $15$  &  $115$  &  $36$  &  $79$  &  $115$  &  $36$  &  $2/1$	& $39/79$		\\
   &  $20$  &  $98$  &  $17$  &  $81$  &  $98$  &  $17$  &  $1/2$ 		& $81/40$	\\
   &  $25$  &  $400$  &  $14$  &  $386$  &  $400$  &  $14$  &  $4/2$ 	& $96/193$	\\
   &  $30$  &  $168$  &  $11$  &  $157$  &  $168$  &  $11$ &  $1$ 	& $157$	\\
 $20$  &  $5$  &  $67$  &  $17$  &  $50$  &  $67$  &  $17$  &  $1$ 	& $50$		\\
 &  $10$  &  $50$ &  $19$  &  $31$  &  $50$  &  $19$  &  $1$ 	& $31$			\\
 &  $15$  &  $97$  &  $42$  &  $55$  &  $97$  &  $42$  &  $1/2$ 	& $55/27$	\\
 &  $20$  &  $34$  &  $8$  &  $26$  &  $34$  &  $8$  &  $1/2$ 	& $26/13$			\\
 &  $25$  &  $289$  &  $10$  &  $279$  &  $289$  &  $10$  &  $2$ 	&	$139$	\\
 &  $30$  &  $109$  &  $22$  &  $87$  &  $109$ &  $22$  &  $1$ 		&	$87$	\\
 $30$  &  $5$  &  $40$  &  $30$  &  $10$  &  $40$  &  $30$  &  $1$ 		&	$10$	\\
  &  $10$ &  $47$  &  $30$  &  $17$  &  $47$  &  $30$  &  $1$ 			&		$17$	\\
  &  $15$  &  $400$  &  $12$  &  $388$  &  $400$  &  $12$  &  $4$ 		&	$97$	\\
  &  $20$  &  $51$  &  $3$  &  $48$  &  $51$  &  $3$  &  $1/3$ 			&	$48/16$	\\
  &  $25$  &  $109/107$  &  $3$  &  $106/104$  &  $108/113$  &  $3$  &  $1/3$ &   -   \\
  &  $30$  &  $400/177$  &  $20$  &  $380/157$ &  $400/374$  &  $20$  &  $2/4$ &     -  \\
 $40$  &  $5$  &  $113$  &  $49$  &  $64$  &  $113$  &  $49$  &  $1$     &     $64$ \\
  &  $10$  &  $400$  &  $97$  &  $303$  &  $400$  &  $97$  &  $1$       &     $303$ \\
  &  $15$  &  $210/211$  &  $10$ &  $200/201$  &  $207/213$  &  $10$  &  $1/2$  & -    \\
  &  $20$  &  $58$  &  $25$  &  $33$  &  $58$  &  $25$  & $1$    &   $33$\\
  &  $25$  &  $134$  &  $1$  &  $133$  &  $134$  &  $1$  &  $1/3$ &    $133/44$ \\
  &  $30$  &  $146$  & $12$  &  $134$  &  $146/155$  &  $12$  &  $4/5$ &   -\\  
\end{tabular}
\caption{The results for the nonlinear benchmark problem~\cite[Sec.~6.3]{Kuratko:2016}. On the left hand side there are the parameters, in the middle the results of Alg.~\ref{Alg:S=Hybrid} with no monitoring of pivots and on the right hand side the results of Alg.~\ref{Alg:S=Hybrid} with the monitoring.}
\label{tab:ODE07HYBandMON}
\end{table}

We demonstrated in Tab.~\ref{tab:ODE01HYBandMON} and~\ref{tab:ODE07HYBandMON} that we can switch between a cheap factorization ($LDL^T$ without pivoting) and Bunch-Parlett. In addition we can minimize the cost of finding the pivots in the Bunch-Parlett method. The monitoring strategy that allows us to reuse the permutation matrices in the Bunch-Parlett method is independent of our application and may be used in other problems as well.

We can compare the results in Tab.~\ref{tab:ODE01HYBandMON} and~\ref{tab:ODE07HYBandMON} with the results in the paper~\cite{Kuratko:2016}, where the preconditioned projected conjugate gradient (PPCG) method~\cite{LuksanVlcek:2001,Nocedal:2006} was used. For the linear benchmark problem the Alg.~\ref{Alg:S=Hybrid} required less iterations of the LS-SQP. However, in the nonlinear case, the results are inconclusive.

All the computations were carried out in Scilab 5.5.2~\cite{Scilab} on a computer Intel(R) Xeon(R) CPU X5680 @ 3.33GHz with the operating system Cent OS 6.8. We used the built in ode solver \emph{ode} in the default settings and the \emph{backslash} operator for solving systems of linear equations in Alg.~\ref{Alg:$LDL^T$ Factorization}--\ref{Alg:S=LDLt} and~\ref{Alg:S=Hybrid}.

%Note that we use the term ``to compute the matrix $P$'' in a sense of finding a permutation of rows and columns of the matrix $K$~\eqref{eq:KKTsystem}. In our implementation of Alg.~\ref{Alg:S=Hybrid} we never construct the matrix $P$ representing a permutation.
%
\section{Conclusion}
\label{sec:Conclusion}
%The motivation to solve the saddle-point system~\eqref{eq:KKTsystem} and study its $LDL^T$ factorization comes from the application in the optimization of dynamical systems~\cite{Kuratko:2016}. We showed in Lemmas~\ref{lem:L_H},~\ref{lem:BLD} and~\ref{lem:L_s} the banded structure of the factor $L$ and proved that the width of bands in $L$ depends only on the parameter $k \in \mathbb{N}$. 

%We explored the use of the $LDL^T$ factorization with no pivoting strategy and ran into problems with the ill-conditioning of the upper left block $H$ of the saddle-point system~\eqref{eq:KKTsystem}. To this end we developed the Alg.~\ref{Alg:S=Hybrid} that switches from $LDL^T$ with no pivoting to the Bunch-Parlett method.

%In SQP one solves a saddle-point system in every iteration, therefore, we reduced the cost of finding the pivots. 
We proposed and tested a pivot monitoring strategy that allows us to reuse and update permutation matrices. Therefore, we reduced the cost of finding the pivots in solving a sequence of a saddle-point systems. Numerical experiments show that this successfully speeds up computation in the frame of dynamical systems optimization. 

The result is a method that is less stable than the unmodified Bunch-Parlett, as shown in Section~\ref{sec:UpdatesBP}. However, practice has shown this is very often not a big concern. For example Ashcraft at al.~\cite[p.~552]{Ashcraft:1998}, observe that for sparse matrices and symmetric factorizations ``very often less stable algorithms appear to perform numerically just as well as more reliable algorithms.'' Another observation is that the unpivoted factorization requires less time and storage than pivoted factorizations with stability guarantees~\cite[p.~552]{Ashcraft:1998}. Our experiments confirm those observations.

Another observation that we find useful is the following. When matrices $K_i$ have fixed structure of nonzero entries and the matrix $P$ remains unchanged, then the structure of nonzero entries in the factor $L$ remains the same. This becomes interesting for memory allocation of sparse matrices. Keeping the same matrix $P$ is a sort of data preprocessing. We can arrange matrices so that the pivots are on the diagonal during the factorization. 
%\section*{Acknowledgement}
%We thank Miroslav Rozlo\v{z}n\'{i}k and Ladislav Luk\v{s}an for comments that greatly improved the manuscript.
%		My bibliography
\bibliographystyle{abbrv}
\bibliography{../bibliography/kuratko}

\end{document}